\documentclass[10 pt, leqno]{amsart}

\usepackage {amsfonts}
\usepackage{amsthm}
\usepackage{amssymb}
\usepackage{latexsym}
\usepackage{amsmath}
\usepackage{mathrsfs}

\theoremstyle{plain}
\newtheorem{tw}{Theorem}[section]

\newtheorem {lem} [tw]{Lemma}

\theoremstyle{definition}

\newcommand{\bc} {\Bbb C}
\newcommand{\bn}{\Bbb N}

\newcommand{\bt}{\Bbb T}
\newcommand{\bz}{\Bbb Z}

\newcommand{\la}{\langle}
\newcommand{\ra}{\rangle}

\newcommand{\ol}{\overline}

\newcommand{\ot}{\otimes}

\newcommand{\wt}{\widetilde}

\newcommand{\alg} {\mathsf{A}}

\newcommand{\clg} {\mathsf{C}}

\newcommand {\hte} {{\textrm{ht}}}

\newcommand {\In} {{\textup{Ind}}}

\newcommand {\End} {{\textrm{End}}}

\newcommand{\tu}{\textup}

\newcommand{\Hil}{\mathsf{H}}

\newcommand{\Ind}{\mathcal{J}}

\newcommand{\Pp}{\mathcal{P}}

\newcommand{\Cun}{\mathcal{O}_N}
\newcommand{\Can}{\mathfrak{C}}

\newcommand{\mlg}{\mathsf{M}}
\newcommand{\vnUHF}{\mlg^{(0)}}
\newcommand{\nlg}{\mathsf{N}}

\newcommand{\UHF}{\mathcal{F}_N}
\newcommand{\nmasa}{\mathcal{C}_N}

\numberwithin{equation}{section}

\begin{document}

\author{Adam Skalski}
{\renewcommand{\thefootnote}{} \footnote{\emph{Permanent address of AS:} Faculty of Mathematics and Computer Science, University
of \L\'{o}d\'{z}, ul. Banacha
 22, 90-238 \L\'{o}d\'{z}, Poland.}}
\address{Department of Mathematics and Statistics,  Lancaster University, Lancaster, LA1 4YF}
 \email{a.skalski@lancaster.ac.uk}

\title{\bf Noncommutative topological entropy of endomorphisms of Cuntz algebras II}

\keywords{Noncommutative topological entropy, Cuntz algebra, polynomial endomorphisms}
\subjclass[2000]{ Primary 46L55, Secondary
37B40}

\begin{abstract}
\noindent A study of noncommutative topological entropy of gauge invariant endomorphisms of Cuntz algebras began in our earlier
work with J.\,Zacharias is continued and extended to endomorphisms which are not necessarily of permutation type. In particular
it is shown that if $\mathsf{H}$ is an $N$-dimensional Hilbert space, $V$ is an irreducible multiplicative unitary on $\mathsf{H}
\otimes \mathsf{H}$ and $F: \mathsf{H} \otimes \mathsf{H} \to \mathsf{H} \otimes \mathsf{H}$ is the tensor flip, then the
Voiculescu entropy of the Longo's canonical endomorphism $\rho_{VF} \in \textup{End}(\mathcal{O}_N)$ is equal to $\log N$.
\end{abstract}

 \maketitle

Noncommutative topological entropy for transformations of $C^*$-algebras introduced by D.\,Voiculescu in \cite{Voic} is an
interesting invariant generalising the classical topological entropy of continuous transformations of compact spaces
(\cite{book}); its value for an endomorphism $\rho$ will be denoted by $\hte(\rho)$. The most studied from the point of view of
the Voiculescu entropy class of transformations is the family of various noncommutative generalisations of classical shifts, in
particular the canonical shift on the Cuntz algebra $\Cun$. In \cite{LMP} together with J.\,Zacharias we computed the Voiculescu
entropy of certain permutation endomorphisms of $\Cun$ generalising the canonical shift.

The basic method of obtaining the lower estimates for the entropy $\hte(\rho)$ of a given endomorphism $\rho$ of a $C^*$-algebra
$\alg$ is based on finding a commutative $C^*$-subalgebra left invariant by $\rho$ on which $\rho$ is induced by a classical
transformation whose entropy can be computed by an application of standard dynamical systems' techniques. Having said that, the
examples of bitstream shifts studied in \cite{book} show that it can happen that \[\hte(\rho) > \hte_c(\rho) =  \sup
\left\{\hte(\rho|_{\clg}): \clg \textrm{ is a commutative } \rho\textrm{-invariant } C^*\textrm{-subalgebra of  } \alg\right\}\]
(see \cite{note}). Each permutation endomorphism of $\Cun$ leaves the so-called canonical masa $\nmasa\subset \Cun$ invariant.
Although the analysis in \cite{LMP} showed that there exist examples when $\hte(\rho) > \hte(\rho|_{\nmasa})$, in each case we
still had $\hte(\rho)=\hte_c(\rho)$ and the actual value of entropy was achieved already on some $\rho$-invariant standard masa
(i.e.\ a masa arising from $\nmasa$ by a change of coordinates). Moreover the only values of entropy obtained explicitly for
endomorphisms of $\Cun$ were equal to $\log k$, where $k \in \bn$.

In view of the above discussion it is interesting to investigate on one hand the existence of invariant standard masas for
endomorphisms of a Cuntz algebra and on the other to seek new ways of establishing lower bounds for Voiculescu entropy for such
endomorphisms. The first problem was studied by J.H.\,Hong, W.\,Szyma\'nski and the author in the recent paper \cite{HS}. The
current work is devoted to making progress in the second. The way forward was suggested by connections between the values of
index of permutation endomorphisms and their entropy discussed in \cite{cs2} and resembling the connections between the index of
a subfactor  of a finite factor and the Connes-St\o rmer entropy of a canonical shift of Ocneanu (see for example \cite{Hiai}).

We begin the note  by introducing the notations and recalling basic statements needed in what follows. In Section 2 we analyse a
class of examples of endomorphisms of $\Cun$ introduced by M.\,Izumi in \cite{Izu1}. Each endomorphism $\rho$ in this class,
although not a permutation endomorphism itself, is at the same time a square root of a permutation endomorphism and a composition
of a permutation endomorphism with a Bogolyubov automorphism $\beta$. Moreover $\hte(\rho)=\frac{1}{2} \log N$ and $\hte (\rho
\circ \beta)=\log N$ (note that in \cite{HS} it was shown that when $N=2$ the corresponding endomorphism leaves no standard masa
invariant). In Section 3 we prove that the computation of the entropy of the Izumi's examples can be interpreted as a special
instance of a general fact: each irreducible multiplicative unitary $V$ on a finite-dimensional Hilbert space $\Hil$
(\cite{Baaj}) leads to an endomorphism of a corresponding Cuntz algebra whose Voiculescu entropy is equal to the logarithm of the
dimension of $\Hil$.

\section{Notations and preliminaries}

Let $N$ in $\bn$ and let $\Cun$ denote the Cuntz algebra, i.e.\ the $C^*$-algebra generated by $N$ isometries with orthogonal
ranges summing to $1$ (\cite{Cuntz1}). The isometries generating $\Cun$ will be usually denoted by $S_1,\ldots,S_N$; composition
of several generating isometries will be expressed via a multi-index type notation (see for example \cite{LMP}). For each $k \in
\bn$ we denote the set of multi-indices of length $k$ by $\Ind_k:=\{(j_1,\ldots,j_k):j_1, \ldots,j_k \in \{1,\ldots,N\}\}$ and
put $\Ind=\bigcup_{k \in \bn} \Ind_k$. There is a well-known $1-1$ correspondence between the unital endomorphisms of $\Cun$
(denoted further by $\End(\Cun)$) and unitaries in $\Cun$, first observed in \cite{Cuntz}: given $U \in \mathcal{U}(\Cun)$ the
associated endomorphism is determined by
\[ \rho_U (S_i) = US_i,\;\;\; i=1,\ldots,N,\]
and conversely given $\rho \in \tu{End}(\Cun)$ its associated unitary is given by
\[U_{\rho}= \sum_{i=1}^n \rho(S_i) S_i^*.\]
If a unitary $U\in \Cun$ is a linear combination of elements of the form
$S_iS_j^*$ ($i,j=1,\ldots,N$), the associated transformation $\rho_U$ is an automorphism, called a \emph{Bogolyubov} automorphism
(its action corresponds to the change of coordinates in the Hilbert space $\bc^N$ underlying the Cuntz algebra in the approach
due to R.\,Longo and J.\,Roberts). In particular the family of automorphisms $\{\rho_{t}: t \in \bt\}$ provides an action of
$\bt$ on $\Cun$, called the gauge action (for Bogolyubov automorphisms we will usually write $\alpha_U$ instead of $\rho_U$).
The fixed point space of this action is denoted by $\UHF$ and is generated as a normed space by the union $\bigcup_{k \in \bn}
\UHF^k$, where $\UHF^k=\tu{Lin}\{S_{J} S_{K}^*: J,K \in \Ind_k\}$. It is easy to check that $\UHF^k$ is isomorphic to
$M_{N^k}\approx M_N^{\ot k}$ and the corresponding embeddings $\UHF^k \subset \UHF^{k+1}$ are compatible with the usual unital
embeddings on the matrix level, so that $\UHF$ is a UHF algebra.  The \emph{canonical masa} (maximal abelian subalgebra) in
$\Cun$  is the algebra $\nmasa$ generated by $\{s_I s_I^*: I \in \Ind\}$; a \emph{standard masa} is a $C^*$-subalgebra of $\Cun$
equal to $\alpha(\nmasa)$ for some Bogolyubov automorphism $\alpha$.

An interesting class of \emph{permutation endomorphisms} of $\Cun$ having a relatively simple combinatorial description was
introduced in \cite{Kawa} and was subsequently studied for example in \cite{Wojtek} and \cite{CS}. Let $\Pp_k$ denote the set of
all permutations of the set $\Ind_k$. The permutation endomorphism given by $\sigma \in \Pp_k$ is the endomorphism associated to
a unitary $U_{\sigma}= \sum_{J \in \Ind_k} S_{\sigma(J)} S_J^*$. The canonical shift on $\Cun$ is the permutation endomorphism
associated with the flip unitary $F=\sum_{i,j=1}^N S_i S_j S_i^* S_j^*$ and denoted further by $\Phi$. Every standard masa is
isomorphic to the algebra of continuous functions on $\Can$, the full shift on $N$ letters (\cite{classical}).

The standard topological entropy of a continuous transformation $T$ of a compact space (\cite{classical}) will be denoted by
$h_{\tu{top}}(T)$.  For the definition and basic properties of the Voiculescu's noncommutative topological entropy of an
endomorphism (or a completely positive map) of a (nuclear) $C^*$-algebra we refer to the original paper \cite{Voic} or to the
monograph \cite{book}. All the information needed to read this note can be also found in \cite{LMP}.

\section{Entropy of an endomorphism coming from a real sector -- the square root of a canonical endomorphism}\label{Izumi}

In \cite{Izu1} M.\,Izumi studied certain explicit examples of endomorphisms of Cuntz algebras motivated by the subfactor theory.
One class of them (Example 3.7 of \cite{Izu1}) was constructed in the following way: let $G$ be a finite abelian group of cardinality $N\geq 2$ with
the (symmetric) duality bracket $\la\cdot, \cdot \ra:G \times G \to \bt$ satisfying the usual conditions ($g,g',h \in G$)
\[ \ol{\la g, h\ra} = \la -g, h\ra, \;\; \sum_{h \in G} \la h, g\ra =
 \begin{cases}0 & \tu{if }  g\neq e \\ N & \tu{if } g = e\end{cases} ,\]
\[ \la g,h \ra \la g',h\ra = \la g+g',h \ra, \;\; \la g, h \ra = \la h, g\ra\]
(the group operation in $G$ will be written  additively). If $G=\bz/_{N\bz}$ one can put $\la k, l\ra:=\exp(\frac{2\pi
i(kl)}{N})$. We will use elements of $G$ as indices of generating isometries in $\Cun$. Define unitaries $U(g) \in \UHF \subset
\Cun$ ($g \in G$) by
\[U(g) = \sum_{h \in G} \la g,h\ra S_h S_h^*,\]
and the endomorphism $\rho\in \tu{End}(\Cun)$ by
\[ \rho (S_g) = \frac{1}{\sqrt{N}} \sum_{h \in G} \la g, h \ra S_h U(g)^*.\]
The endomorphism $\rho$ was studied in detail in \cite{HS}. It does not leave $\nmasa$ invariant. The unitary associated with
$\rho$ is equal to
\begin{equation} V= \frac{1}{\sqrt{N}} \sum_{g,h,l \in G} \la g,h-l \ra S_h S_l S_l^* S_g^*.\label{Izunit}\end{equation}

Define for each $h \in G$
\[\wt{S}_h = \frac{1}{\sqrt{N}} \sum_{a \in G} \la h, a\ra S_a\]
and let $\beta \in \tu{Aut}(\Cun)$ be given by
\[\beta(S_h) = \wt{S}_h, \;\;\; h \in G.\]
It is easy to see that $\beta$ is a Bogolyubov automorphism.

\begin{lem} \label{Izuent0}
The endomorphism $\rho':=\rho \circ \beta$ is a permutation endomorphism of $\Cun$ given by the formula (notations as above)
\[ \rho'(S_h) = \sum_{g  \in G} S_g S_{h+g} {S_{h+g}}^*, \;\;\; h \in G.\]
Moreover  $\hte(\rho') = \log N$.
\end{lem}

\begin{proof}
The first statement is shown in \cite{HS}.  As the endomorphism $\rho'$ is induced by a unitary in $\UHF^2$, Theorem 2.2 of
\cite{LMP} implies that $\hte(\rho') \leq \log N$. As $\rho'$ is a permutation endomorphism, it leaves $\nmasa$ invariant and
$\rho'|_{\nmasa}$ is induced by a continuous transformation $T_{\rho'}$ of $\Can$. Note that we have ($h,k \in G$)
\[ \rho'(S_h S_h^*) = \sum_{g \in G} S_g S_{h+g} {S_{h+g}}^* S_g ^*,\]
\[ \rho' (S_h S_k S_k^* S_h^*) = \sum_{g \in G} S_g S_{h+g} S_{k+h+g} S_{k+h+g}^* S_{h+g}^* S_{g}^*,\]
and so on. The analysis of the transformations on cylinder sets (see \cite{LMP} or \cite{Wojtek}) implies that $T_{\rho'}$ is
given by the formula
\[ (T_{\rho'}(w))_k = w_{k+1}  - w_k, \;\;\; w:=(w_n)_{n=1}^{\infty}\in \Can, k \in \bn.\]
Hence  an application of Lemma 3.2 of \cite{LMP} yields that $h_{\tu{top}}(T_{\rho'}) = \log N$ so that $\hte (\rho') \geq
\hte(\rho'|_{\nmasa})= h_{\tu{top}}(T_{\rho'})= \log N$.
\end{proof}

\begin{tw} \label{Izumient}
The Voiculescu entropy of $\rho$ is equal to $\frac{1}{2}\log N$.
\end{tw}
\begin{proof}
It is easily checked that $ \hte(\rho)=\frac{1}{2}\hte(\rho^2)$. Write $\gamma:=\rho^2$. It suffices to show that $\hte(\gamma) =
\log N$. This will follow from the general result in Theorem \ref{genthm}, but here we can provide a direct proof, as $\gamma$ is
a permutation endomorphism given by the formula ($g \in G$)
\[\gamma(S_g)  = \sum_{k \in G}   S_k      S_{g+k} S_{k}^*\]
(see \cite{HS}).  As the associated unitary $V_{\gamma}:=\sum_{g,h \in G}   S_g      S_{h+g} S_{g}^* S_h^*$ belongs to $\UHF^2$,
Theorem 2.2 of \cite{LMP} implies that $\hte(\gamma) \leq \log N$.

Examine the action of $\gamma$ on $\nmasa$. For each $n\in \bn$ and $g_1, \ldots,g_n \in G$
\[ \gamma(S_{g_1} \cdots S_{g_n} S_{g_n}^* \cdots S_{g_1}^*) = \sum_{h \in G} S_h S_{g_1+h} \cdots S_{g_n+h} S_{g_n+h}^* \cdots
S_{g_1+h}^* S_h^*,\] so that once again analysing the cylinder sets we see that $\gamma|_{\nmasa}$ is induced by the continuous
transformation defined by
\[ (T_{\gamma} (w))_k = w_1 + w_{k+1}, \;\; w =(w_n)_{n=1}^{\infty}\in \Can, k \in \bn.\]
Thus appealing to Lemma 3.2 in \cite{LMP} yields $\hte(\gamma|_{\nmasa})= h_{\tu{top}}(T_{\gamma}) = \log N.$ This ends the
proof.
\end{proof}

Note that Lemma \ref{Izuent0} and Theorem \ref{Izumient} yield an example of an endomorphism $\rho$ of $\Cun$ and a Bogolyubov
automorphism $\beta$ such that $\hte(\rho \circ \beta) \neq \hte(\rho)$ (although, as follows from \cite{free}, the Voiculescu
entropy of each Bogolyubov automorphism of $\Cun$ is equal $0$).

Let $G=\bz/_{2\bz}=\{0,1\}$ with the natural duality bracket ($\la 1, 1 \ra = -1$, all other brackets take value $1$). The Izumi
endomorphism discussed above is then given by \begin{equation} \label{G2}\rho(S_0) = \frac{1}{\sqrt{2}} (S_0 + S_1), \;\;\;
\rho(S_0) = \frac{1}{\sqrt{2}} (S_0 S_0 S_0^* + S_1 S_1 S_1^*  - S_1 S_0 S_0^* - S_0 S_1 S_1^*).\end{equation}

In Section 6 of \cite{HS} we showed that $\rho$ defined via the formulas in \eqref{G2} does not leave any standard masa
invariant. This naturally leads to the following closely connected questions: can one characterise masas in $\Cun$ left invariant
by $\rho$? Do we have $\hte(\rho)= \hte_c(\rho)$?

It would be very interesting to investigate the entropy of other examples of real sectors given in \cite{Izu1}. This would
require completely new methods even for obtaining upper estimates, as Theorem 2.2 of \cite{LMP} applies only to the endomorphisms
associated to unitaries in $\UHF$ and other examples of Section 3 of \cite{Izu1} are not of this type.

\section{Entropy of canonical endomorphisms associated to multiplicative unitaries}

Consider again the endomorphism $\rho$ associated to the unitary defined in \eqref{Izunit} via a symmetric duality bracket on a
finite abelian group $G$  discussed in the last section. Let $\tau$ be the faithful trace on $\UHF$ and let $\phi=\tau \circ E$,
where $E:\Cun \to \UHF$ is the canonical conditional expectation (given by integrating the gauge action). It follows from Lemma
2.1 of \cite{Longosub} that the endomorphism $\rho$ preserves $\phi$, so also extends to an endomorphism of
$\mlg:=\pi_{\phi}(\Cun)''$, where $\pi_{\phi}$ denotes the GNS representation with respect to $\phi$ (we will denote the
extension by $\wt{\rho}$). It follows from the easily checked condition in Corollary 4.3 of \cite{ConP} that $\wt{\rho}$ is
irreducible, i.e.\ $\wt{\rho}(\mlg)'\cap \mlg =\bc I_{\mlg}$.

As discussed in Section 2, in \cite{Izu1} it is observed (as a consequence of results in \cite{Longo} and \cite{Wat}) that $\rho$
is a restriction of a `square root' of a canonical endomorphism of $III_{\frac{1}{N}}$-factor generated by $\Cun$ in the GNS
representation with respect to the state $\phi$. It follows from Proposition 2.5 in \cite{Izu1} that if we consider the
conditional expectation $E_{\rho}:\Cun \to \rho(\Cun)$ defined by
\[ E_{\rho}(x) = \rho (S_e^* \rho(x) S_e),\]
where $e\in G$ is the neutral element, then $\In \, E_{\rho} = N$ and it is a minimal index in the sense of Kosaki
(\cite{Kosaki}). Moreover the endomorphism $\gamma=\rho^2$ is related to the left regular representation of the group $G$.
Indeed, it is easy to check that
\begin{equation}\gamma^2 = \Phi \circ \gamma,\label{Cuntzformula}\end{equation}
which due to Proposition 2.1 of \cite{Cuntzmunit} is equivalent to the fact that the unitary associated with $\gamma$ is a
product of a multiplicative unitary on $\bc^n \ot \bc^n$ (\cite{Baaj}) and the flip unitary $F=\sum_{g,h \in G} S_h S_g S_h^*
S_g^*$.

As suggested by the above discussion, there is a connection between the entropy computation in Section \ref{Izumi}, index values
and the fact that $\rho^2$ is related to the left regular representation of a finite group. Indeed,
Theorem \ref{Izumient} may be viewed as a special
instance of a general entropy result related to interaction between finite-dimensional Kac algebras and index for subfactors
stated in the following theorem.

\begin{tw}\label{genthm}
Let $V$ be an irreducible multiplicative unitary on $\Hil \ot \Hil$, where $\Hil$ is an $N$-dimensional Hilbert space; view it as
a matrix in $M_N \ot M_N$ and further via the usual isomorphism $M_N\ot M_N \approx \UHF^2\subset \Cun$ as a unitary in $\Cun$.
Let $F$ be the  flip unitary in $\UHF^2$. The topological entropy of the endomorphism of $\Cun$ associated with $VF$ is equal to
$\log N$.
\end{tw}

\begin{proof}
Denote $R=VF$, $\gamma:=\rho_{R}$. As $R\in \UHF^2$, Theorem 2.2 of \cite{LMP} implies that $\hte(\gamma) \leq \log N$.

As mentioned above, Lemma 2.1 of \cite{Longosub} implies that $\gamma$ preserves $\phi$, so $\gamma$ extends to an endomorphism
of $\mlg:=\pi_{\phi}(\Cun)''$, denoted by $\wt{\gamma}$. Irreducibility of $V$ implies that $\wt{\gamma}$ is a canonical
endomorphism for an irreducible inclusion of factors $\nlg \subset \mlg$ (Corollary 4.3, \cite{Longosub}), where $\nlg$ is the
fixed point algebra for the coaction associated to $V$ (\cite{Longosub}, \cite{Cuntzmunit}). By Proposition 3.1 in
\cite{Longosub} $\In(\wt{\gamma})=N^2$.

 Let $\vnUHF:=\pi_{\phi}(\UHF)''= \pi_{\phi}\left(\bigcup_{n \in \bn} \UHF^n\right)''$.
 Due to Proposition 4.7 in \cite{Longosub}, $\vnUHF$  is equal to the strong closure of the algebra $\bigcup_{n\in \bn} \wt{\gamma}^n(\mlg)'\cap\mlg.$
 As $V\in \UHF$, the endomorphism
$\wt{\gamma}$ leaves $\vnUHF$ invariant. By Theorem 3.2.2 (ii) in \cite{book} we have $\tu{h}_{\phi} (\wt{\gamma}) =
\tu{h}_{\phi} (\gamma)$, where $h_{\omega}(\alpha)$ denotes the CNT entropy of an endomorphism $\alpha$ preserving a state
$\omega$ (see \cite{book} for the precise definitions -- although Theorem 3.2.2 (ii) is stated for the automorphisms, its proof
is valid also when $\rho$ is just an endomorphism). As  by Theorem 6.2.2 (ii)  in \cite{book} $\hte(\gamma) \geq \tu{h}_{\phi}
(\gamma)$, to finish the proof it suffices to show that $\tu{h}_{\phi} (\wt{\gamma}) \geq \log N$. Note that as the conditional
expectation $E$ extends to a $\phi$-preserving normal conditional expectation $\wt{E}:\mlg \to \vnUHF$, by Theorem 3.2.2 (v) in
\cite{book} we have $\tu{h}_{\phi} (\wt{\gamma})\geq \tu{h}_{\phi} (\wt{\gamma}|_{\vnUHF})$.

Let $\nlg^{(0)}:=\nlg \cap \vnUHF$. Note that because both $\gamma$ and canonical shift $\Phi$ commute with the gauge action, we
have $\nlg^{(0)}= \wt{E}(\nlg)$. As the canonical expectation $E_{\gamma}:\mlg \to \wt{\gamma}(\mlg)$ discussed in Lemma 4.6 in
\cite{Longosub} preserves the trace on $\vnUHF$, due to Lemma 7.3.5 in \cite{Choda} we see that $\nlg^{(0)}$ is a subfactor of
$\vnUHF$; Proposition 7.3.6 of the same paper\footnote{Note that in Lemma 7.2.1 in \cite{Choda} the operator $V$ should be
defined as $V=\frac{1}{\lambda} \gamma(e)fe$ and is only a \emph{partial} isometry -- this does not affect further reasoning and
the main results of that paper remain valid.} implies $[\mlg^{(0)}, \nlg^{(0)}] = (\In(\wt{\gamma}))^{\frac{1}{2}}=N$. Further we
can use the observation in Theorem 7.3.7 of \cite{Choda} (see also Corollary 4.3 in \cite{Longosub})
 that $\wt{\gamma}|_{\vnUHF}$ is conjugate to
the Ocneanu's canonical shift associated with the inclusion $\nlg^{(0)} \subset \mlg^{(0)}$; note that now we are in the
framework of finite factor inclusions. Corollary 4.6 of \cite{Hiai} gives then $\tu{h}_{\phi} (\wt{\gamma}|_{\vnUHF})=
\log([\mlg^{(0)}, \nlg^{(0)}])= \log N$, provided the inclusion $\nlg^{(0)} \subset \mlg^{(0)}$ is extremal and strongly
amenable. Extremality follows from the equality  $(\nlg^{(0)}) ' \cap \mlg^{(0)} = \bc I_{\mlg}$, which itself is a consequence
of Theorem 6.6 in \cite{dimension} and Proposition 3.2 of \cite{Longosub}. By Theorem 1 in \cite{Popa} strong amenability of the
inclusion in question will follow if we can only show that it has finite depth, as $\vnUHF$ is a hyperfinite factor. As the Jones
tunnel corresponding to the inclusion $\nlg^{(0)} \subset \mlg^{(0)}$ is given by
\[ \cdots \subset \gamma(\nlg^{(0)}) \subset \gamma(\vnUHF) \subset \nlg^{(0)} \subset \vnUHF \subset \cdots,\]
it suffices to show that $\gamma(\nlg^{(0)})' \cap \vnUHF$ is a factor. The earlier observation that $(\nlg^{(0)}) ' \cap
\mlg^{(0)} = \bc I_{\mlg}$ and a suitably adapted argument from Corollary 3.3 of \cite{Longosub} delivers precisely that (one can
show that $\gamma(\nlg^{(0)})' \cap \vnUHF=\UHF^1$).
\end{proof}

Conceptually the reason for which we get $\hte(\rho_{VF})= \log N$ is that $\rho_{VF}$ is indeed a map closely related to the
canonical shift on $\Cun$, as is suggested by the formula \eqref{Cuntzformula}; several instances of such analogies can be found
in \cite{dimension}.

\vspace*{0.5cm}

\noindent \emph{Acknowledgment.} The work on this note was started during a visit of the author to University of Tokyo in
October-November 2009 funded by a JSPS Short Term Postdoctoral Fellowship.

\end{document}